\documentclass{amsart}
\usepackage{amsmath}
\usepackage{color}
\usepackage{mathtools}

\newcommand{\ol}{\overline}
\newcommand{\ul}{\underline}

\newtheorem{theorem}{Theorem}[section]
\newtheorem{lemma}[theorem]{Lemma}
\newtheorem{proposition}[theorem]{Proposition}

 \theoremstyle{definition}

\theoremstyle{remark}

\numberwithin{equation}{section}

\begin{document}

\title[Degenerate Hessian equations]
{Second order estimates for convex solutions of degenerate $k$-Hessian equations}

\author{Heming Jiao}
\address{School of Mathematics and Institute for Advanced Study in Mathematics, Harbin Institute of Technology,
         Harbin, Heilongjiang 150001, China}
\email{jiao@hit.edu.cn}

\author{Zhizhang Wang}
\address{School of Mathematical Sciences, Fudan University, Shanghai, China}
\email{zzwang@fudan.edu.cn}
\thanks{The first author is supported by the NSFC (Grant No. 11871243). Research of the second author is sponsored by Natural Science  Foundation of Shanghai, No.20JC1412400,  20ZR1406600 and supported by NSFC Grants No. 11871161,12141105.}

\begin{abstract}

The $C^{1,1}$ estimate of the Dirichlet problem for degenerate $k$-Hessian equations with non-homogenous boundary conditions is an open problem, if
the right hand side function $f$ is only assumed to satisfy $f^{1/(k-1)} \in C^{1,1}$. In this paper, we solve this problem for convex solutions defined in the strictly convex bounded domain.


\end{abstract}

\maketitle

\section{Introduction}
Suppose $u$ is some function defined in a bounded domain $\Omega\subset \mathbb{R}^n$ and $\varphi$ is some given function defined on the boundary $\partial\Omega$. In this paper, we concern the  Dirichlet problem
of degenerate $k$-Hessian equations with non-homogenous boundary functions ($k \geq 2$)
\begin{equation}
\label{1-1}
\left\{ \begin{aligned}
   \sigma_k \big(\lambda(D^{2} u)\big) & = f  \;\;\mbox{ in }~ \Omega, \\
                 u &= \varphi  \;\;\mbox{ on }~ \partial \Omega,
\end{aligned} \right.
\end{equation}
where $f \geq 0$ in $\Omega$ and
$\sigma_{k}$ are the elementary symmetric functions
\[
\sigma_{k} (\lambda) = \sum_ {i_{1} < \ldots < i_{k}}
\lambda_{i_{1}} \ldots \lambda_{i_{k}},\ \ k = 1, \ldots, n,
\]
and $\lambda(D^2u)$ is the eigenvalue vector of the hessian $D^2u$.
The Poisson equation
and Monge-Amp\`{e}re equation
fall into the form of \eqref{1-1} as $k = 1$ and $k = n$ respectively. We call a function
$u \in C^2 (\Omega)$ is $k$-convex if $\lambda (D^2 u) \in \ol \Gamma_k$ in $\Omega$, where $\Gamma_k$ is
the G{\aa}rding's cone
\[
\Gamma_{k} = \{\lambda \in \mathbb{R}^{n}: \sigma_{j} (\lambda) > 0, j = 1, \ldots, k\}.
\]
A bounded domain $\Omega$ in $\mathbb{R}^n$ is called uniformly $(k-1)$-convex, if there is a positive constant
$K$ such that for each $x \in \partial \Omega$,
\[
(\kappa_1 (x), \ldots, \kappa_{n - 1} (x), K) \in \Gamma_k,
\]
where $\kappa_1 (x), \ldots, \kappa_{n - 1} (x)$ are the principal curvatures of $\partial \Omega$ at
$x$.

The central issue of the degenerate equations is the existence of $C^{1,1}$ solution, which major needs  the a prior $C^{1,1}$ estimates of the solutions. Unlike the non degenerate equations, the establishment of $C^{1,1}$ estimate always requires some regularity of $f$ near the set of $f=0$. In \cite{GL1} and \cite{GL2}, Guan-Li have studied the degenerate Weyl problem and degenerate Gauss curvature measure problem, which both are degenerate Monge-Amp\`ere type equations. They found that the conditions
\begin{eqnarray}\label{cond1}
\Delta (f^{1/(n-1)})\geq -A \text{ and } |D (f^{1/(n-1)})|\leq A
\end{eqnarray}
for some constant $A$, are sufficient to get the global $C^{1,1}$ bound. Soon after, for the Dirichlet problem of the degenerate Monge-Amp\`ere equations, Guan \cite{Guan97} found that the above conditions also are sufficient to get the $C^{1,1}$ boundary estimates, if we only consider the homogenous boundary problem. For the non homogenous boundary problem,  Guan-Trudinger-Wang \cite{GTW99}  established the $C^{1,1}$ boundary estimate, if one require
\begin{eqnarray}\label{cond2}
f^{1/(n-1)}\in C^{1,1}(\overline{\Omega}).
\end{eqnarray}
It is not difficult to see, \eqref{cond2} implies \eqref{cond1}.

In view of the above results, a natural and interesting question is that can we establish the $C^{1,1}$ estimates for $k$-convex solutions of \eqref{1-1}, including boundary estimates  and global estimates, by using the condition
\begin{eqnarray}\label{condition}
f^{1/(k-1)}\in C^{1,1}(\overline{\Omega}).
\end{eqnarray}
Note that if $k=n$,  \eqref{condition} is \eqref{cond2}.  This question is first proposed by Ivochkina-Trudinger-Wang \cite{ITW2004},
which is not solved until now.

Let's review some related research of the above question.
Write $\widetilde{f} := f^{1/(k-1)}$. Dong \cite{Dong06} has considered the homogenous Dirichlet problem \eqref{1-1} with homogenous boundary condition $\varphi \equiv 0$. He established
the $C^2$ estimates by \eqref{condition} and
\begin{equation}
\label{strong}
|D \widetilde{f}| \leq C \widetilde{f}^{1/2} \mbox{ on } \ol \Omega
\end{equation}
for some positive constant $C$. In \cite{JW}, by using \eqref{condition}, the authors obtain the  $C^{1,1}$ estimate  for  the Dirichlet problem of degenerate $k$-curvature equations with homogenous boundary condition, which can be view as a generalization of \cite{Dong06}.
If we require $f^{1/k} \in C^{1,1} (\ol \Omega)$, which is a little stronger than \eqref{condition}, for non homogenous boundary problem,
the $C^{1, 1}$ regularity has been established by Krylov \cite{K94a, K94b, K95a, K95b}, seeing an alternative proof
by Ivochkina-Trudinger-Wang \cite{ITW2004}.
In \cite{Wang95},  Wang gave an example, which shows that the condition \eqref{condition} is optimal for the solutions of $k$-Hessian equations being in $C^{1,1}(\overline{\Omega})$. Therefore, condition \eqref{condition} should be sharp.
For more reference, the reader may see \cite{CNS86, Guan98, WX14}  and the reference therein.

The main results of this paper is to give the boundary estimate for convex solutions of \eqref{1-1}  with non homogenous boundary functions, only using condition \eqref{condition}. Note that, for $k$-Hessian equations, the assumption of convex  solutions has been used in \cite{GRW,GQ, MSY}.  Our result partially answers Ivochkina-Trudinger-Wang's question.

\begin{theorem}
\label{conj2}
Suppose $\Omega$ is uniformly convex with $\partial \Omega \in C^{3, 1}$, $\varphi \in C^{3,1} (\partial \Omega)$,
$f > 0$ in $\Omega$ and \eqref{condition}.
Then any convex solution $u \in C^3 (\Omega) \cap C^2 (\ol \Omega)$
to the Dirichlet problem \eqref{1-1} satisfies the estimates
\begin{equation}
\label{main-estimate1}
\max_{\partial \Omega}|D^2 u| \leq C,
\end{equation}
where the positive constant $C$ depends on $n$, $k$, $\Omega$, $|\varphi|_{C^{3,1} (\partial \Omega)}$ and $|f^{1/(k-1)}|_{C^{1,1} (\ol \Omega)}$
but is independent of the lower bound $\inf_\Omega f$.
\end{theorem}
Theorem \ref{conj2} can be regarded as a generalization of the main results in \cite{GTW99}. The major difficulty to prove \eqref{main-estimate1} is
 the estimates for double normal derivatives which we establish in two
steps. One is Lemma \ref{cor-lemma}, which generalizes Lemma 3.1 of \cite{GTW99}.
We utilize the idea of Trudinger \cite{T}, then we can obtain the same lower bound of tangential-tangential derivatives only assuming that the solutions are $k$-convex for $k$-Hessian equations.
Our method is completely different from \cite{GTW99}.
The other one is an estimate of mixed tangential-normal derivatives in terms of  tangential-tangential derivatives, namely, Lemma \ref{lemma-2}. Similar estimate was
 proved in \cite{GTW99} for Monge-Amp\`{e}re equation,
where the special structure of Monge-Amp\`{e}re equation plays a key role to obtain the bound of the tangential-normal derivatives. Unfortunately, the $k$-Hessian
equations with $2 \leq k \leq n-1$ do not possess such structure. Therefore, we need some new idea to reestablish these estimates without using the affine transformation, which is our novelty of Section 6 and Section 7.

Since the constant in \eqref{main-estimate1} is independent of $\inf_\Omega f$, by approximation and
the $C^1$ estimates, global $C^2$ estimates in section 2, section 3,
we obtain the following $C^{1,1}$ estimates for degenerate $k$-Hessian equations.
\begin{theorem}
\label{conj1}
Suppose $\Omega$ is uniformly convex with $\partial \Omega \in C^{3, 1}$, $\varphi \in C^{3,1} (\partial \Omega)$,
$f \geq 0$ in $\Omega$ and \eqref{condition}.
Then any convex solution $u \in C^4 (\Omega) \cap C^2 (\ol \Omega)$
to the Dirichlet problem \eqref{1-1} satisfies the estimates
\begin{equation}
\label{main-estimate}
|u|_{C^{1,1} (\ol \Omega)} \leq C,
\end{equation}
where the positive constant $C$ depends on $n$, $k$, $\Omega$, $|\varphi|_{C^{3,1} (\partial \Omega)}$ and $|f^{1/(k-1)}|_{C^{1,1} (\ol \Omega)}$.
\end{theorem}

The rest of the paper is organized as follows. In Section 2, we establish the \emph{a priori} $C^1$ estimates. Section 3
is devoted to the maximal principle for second order derivatives. We concern the boundary estimates of the pure tangential derivatives and mixed
second order derivatives  in Section 4. In Section 5, we prove a lower bound of the pure tangential derivatives on the boundary.
In the last two sections, we obtain an upper bound of the tangential-normal derivatives in terms of the pure tangential derivatives for convex solutions.

\section{$C^1$ estimates}

In Section 2 to Section 5, we assume  that $u$ is the $k$-convex solution to \eqref{1-1}.
We establish the $C^1$ estimates for $u$ in this section.
By using the $(k - 1)$-convexity of the domain $\Omega$, we can construct
a subsolution $\ul u$ to \eqref{1-1}  as \cite{CNS},
\begin{equation}
\label{2-1}
\left\{ \begin{aligned}
   \sigma_k \big(\lambda(D^{2} \ul u)\big) & \geq f  \;\;\mbox{ in }~ \Omega, \\
                \ul u &= \varphi  \;\;\mbox{ on }~ \partial \Omega.
\end{aligned} \right.
\end{equation}
Let $h$ be the harmonic function in $\Omega$ with $h = \varphi$ on $\partial \Omega$.
By the comparison principal, we have
\[
\ul u \leq u \leq h \mbox{ in } \Omega \mbox{ and } \ul u = u = h \mbox{ on } \partial \Omega
\]
Thus, we  have
\begin{equation}
\label{C1-1}
\sup_{\ol \Omega} |u| + \sup_{\partial \Omega} |D u| \leq C,
\end{equation}
where the constant $C > 0$ depends only on $|\ul u|_{C^1 (\ol \Omega)}$ and $|h|_{C^1 (\ol \Omega)}$. Define
\[
S_k [r] = \sigma_k (\lambda (r))
\]
for a symmetric matrix $r = \{r_{ij}\}$ with $\lambda (r) \in \Gamma_k$ and
\[
S^{ij}_k = \frac{\partial S_k [D^2u]}{\partial r_{ij}}.
\]
By the Newton-Maclaurin inequality, we have
\begin{equation}
\label{GC2-9}
\sum_i S^{ii}_k = (n - k + 1) S_{k - 1} \geq c_0 S_k^{1 - 1/(k-1)} S_1^{1/(k-1)}
\end{equation}
for some positive constant $c_0$ depending only on $n$ and $k$.
\begin{proposition}
\label{C1-3}
Assume that
\begin{equation}
\label{C1-8}
|D f(x)|\leq A f^{1-\frac{1}{k-1}}(x)
\end{equation}
holds for  some positive constant  $A$ and any $x\in \Omega$.
Then there exists positive constant $C$ depending only on $n$, $k$, $A$
and $\Omega$ such that
\begin{equation}
\label{C1-4}
\sup_{\ol \Omega} |D u| \leq C (1 + \sup_{\partial \Omega} |D u|).
\end{equation}
\end{proposition}
\begin{proof}
We may assume $0 \notin \ol \Omega$ and $|x|^2 \geq \delta_0 > 0$ for all $x \in \ol \Omega$.
Suppose
\[
W := \max_{x \in \ol \Omega, \xi \in \mathbb{S}^n} \{u_\xi + B |x|^2\}
\]
is attained at an interior point $x_0 \in \Omega$ and $\xi_0 \in \mathbb{S}^n$, where $B$
is a positive constant sufficiently large to be determined. We may assume $\xi_0$ is in the
direction $x_1$ by rotating the coordinates.
We have, at $x_0$,
\begin{equation}
\label{C1-9}
\sum_{l=1}^n u^2_{1l} = 4 B^2 \sum_l x^2_l = 4 B^2 |x|^2 \geq 4 B^2 \delta_0.
\end{equation}
By \eqref{C1-8} and \eqref{GC2-9}, at $x_0$, we have
\begin{equation}
\label{C1-5}
\begin{aligned}
0 \geq & S^{ij}_k (u_1 + B |x|^2)_{ij} = f_1 + 2 B \sum_i S^{ii}_k\\
  \geq & - A f^{1 - 1/(k-1)} + 2 c_0 B S_1^{1/(k-1)} S_k^{1 - 1/(k-1)}\\
     = & - A f^{1 - 1/(k-1)} + 2 c_0 B (\Delta u)^{1/(k-1)} f^{1 - 1/(k-1)}.
\end{aligned}
\end{equation}
Since $\lambda (D^2 u) \in \Gamma_k \subset \Gamma_2$, we have
\[
\begin{aligned}
0 < 2 \sigma_2 (\lambda (D^2 u)) = \,& 2 \sum_{1 \leq i < j \leq n} (u_{ii} u_{jj} - u_{ij}^2)\\
   = \,& (\Delta u)^2 - \sum_i u_{ii}^2 - \sum_{i \neq j} u_{ij}^2.
\end{aligned}
\]
We have, by \eqref{C1-9},
\[
(\Delta u)^2 \geq \sum_l u^2_{1l} \geq 4 B^2 \delta_0.
\]
Thus, in view of  \eqref{C1-5}, we derive
\[
0 \geq - A f^{1 - 1/(k-1)} + 2 c_0 B (4 B^2 \delta_0)^{1/(k-1)} f^{1 - 1/(k-1)} > 0
\]
provided $B$ is sufficiently large which is a contradiction.
Then we conclude that $W$ is attained on the boundary when choosing
$B$ sufficiently large and
\eqref{C1-4} holds.
\end{proof}
The $C^1$ estimate has been established in terms of \eqref{C1-1} and \eqref{C1-4}.

\section{The maximal principle for second order derivatives}

In this section, we prove
\begin{proposition}
\label{GC2-2}
Suppose \eqref{C1-8} and
\begin{equation}
\label{GC2-10}
\inf\{\Delta f^{\frac{1}{k-1}}(x)\}\geq -\frac{A}{k-1}
\end{equation}
hold for  some positive constant  $A$ and any $x\in \Omega$.
Then there exists a positive constant $C$ depending on $n$, $k$, $A$ and $\Omega$ such that
\begin{equation}
\label{GC2-3}
\sup_{\ol \Omega} |D^2 u| \leq C (1 + \sup_{\partial \Omega} |D^2 u|).
\end{equation}
\end{proposition}
To prove Theorem \ref{GC2-2}, we need the following lemma which is Lemma 3.2 in \cite{GLL12}.
\begin{lemma}
\label{GLL}
Let $\alpha = \frac{1}{k-1}$. If $\lambda (D^2 u) \in \Gamma_k$, then
\begin{equation}
\label{GLL-1}
\sum_i S_k^{pq, rs} u_{pqi} u_{rsi} \leq - S_k\sum_h\left[\frac{(S_k)_h}{S_k}-\frac{(S_1)_h}{S_1}\right]
\left[(\alpha-1)\frac{(S_k)_h}{S_k}-(\alpha+1)\frac{(S_1)_h}{S_1}\right],
\end{equation}
where
\[
S^{pq, rs}_k = \frac{\partial^2 S_k [D^2u]}{\partial r_{pq} \partial r_{rs}}.
\]
\end{lemma}
Lemma \ref{GLL} comes from the concavity of $\big(\frac{\sigma_k}{\sigma_1}\big)^{1/(k-1)}$
in $\Gamma_k$.
\begin{proof}[Proof of Propsotion \ref{GC2-2}]
We consider the test function
\[
H=\Delta u+\frac{B}{2}|x|^2,
\]
where $B$ is a positive undetermined  constant.
First, by differentiating the equation \eqref{1-1}, we have
\begin{equation}
\label{GC2-11}
S_k^{ij}(\Delta u)_{ij} + \sum_i S_k^{pq,rs}u_{pqi}u_{rsi} = \Delta f,
\end{equation}
which implies
\begin{equation}
\label{GC2-12}
S_k^{ij} H_{ij} = -\sum_i S_k^{pq,rs}u_{pqi}u_{rsi} + B (n - k + 1) S_{k-1} + \Delta f.
\end{equation}
Suppose $H$ attains its maximum at an interior point $x_0 \in \Omega$.
Write $\alpha := \frac{1}{k-1}$. We have, at $x_0$,
\begin{equation}
\label{GC2-14}
0 = H_i = (\Delta u)_i + B x_i
\end{equation}
for each $1 \leq i \leq n$ and
therefore, by \eqref{GLL-1}, \eqref{GC2-9}, \eqref{GC2-10} and \eqref{C1-8}, we get
\begin{equation}
\label{GC2-13}
\begin{aligned}
0 \geq & - \sum_i S_k^{pq,rs} u_{pqi}u_{rsi} + B (n - k + 1) S_{k-1} + \Delta f\\
  \geq & S_k\sum_h\left(\frac{(S_k)_h}{S_k}-\frac{(S_1)_h}{S_1}\right)
\left((\alpha-1)\frac{(S_k)_h}{S_k}-(\alpha+1)\frac{(S_1)_h}{S_1}\right)\\
       & +(c_0 B S_1^{\alpha}-A) f^{1-\alpha} + (1-\alpha)\frac{|\nabla f|^2}{f}\\
     = & (\alpha - 1)\frac{|\nabla f|^2}{f} - 2\alpha \frac{\nabla \Delta u}{\Delta u} \cdot \nabla f
         + (1 + \alpha) f \frac{|\nabla \Delta u|^2}{(\Delta u)^2}\\
       & + (c_0 B (\Delta u)^{\alpha} - A) f^{1-\alpha} + (1-\alpha)\frac{|\nabla f|^2}{f}\\
     = & \frac{2 \alpha B x \cdot \nabla f}{\Delta u} +(1 + \alpha) f \frac{B^2 |x|^2}{(\Delta u)^2}
         + (c_0 B (\Delta u)^{\alpha} - A) f^{1-\alpha}\\
  \geq & \left(c_0 B (\Delta u)^{\alpha} - A - \frac{C B}{\Delta u}\right)f^{1-\alpha},
\end{aligned}
\end{equation}
where $C$ is some constant only depending on $\Omega,n,k$ and $A$.
Thus, for sufficient large $B$ and $\Delta u$, we have the desired estimates.
\end{proof}
If there is no degenerate point on $\partial\Omega$, we have the usual boundary estimate as \cite{CNS}. Therefore, using Proposition \ref{C1-3}  and Proposition \ref{GC2-2},  we can prove
\begin{theorem}
\label{cor1}
Suppose $\Omega$ is $(k-1)$-convex with $\partial \Omega \in C^{3, 1}$, $\varphi \in C^{3,1} (\ol \Omega)$,
$f \geq 0$ in $\Omega$, $f > 0$ on $\partial \Omega$ and \eqref{C1-8}, \eqref{GC2-10} hold for some constant $A>0$.
Then the estimate \eqref{main-estimate} holds for any
 $k$-convex  $C^{1,1}$ solution $u$ of \eqref{1-1}.
\end{theorem}

\section{Estimates for mixed tangential-normal derivatives}

In this and the following sections, we derive the second order boundary estimates.
For any point $x_0 \in \partial \Omega$, we may assume that $x_0$ is the origin and that the
positive $x_n$-axis is in the interior normal direction to $\partial \Omega$ at the origin.
Suppose near the origin, the boundary $\partial \Omega$ is given by
\begin{equation}
\label{BC2-1}
x_n = \rho (x') = \frac{1}{2} \sum_{\alpha, \beta < n} B_{\alpha\beta} x_\alpha x_\beta + O (|x'|^3),
\end{equation}
where $x' = (x_1, \ldots, x_{n-1})$ and $B_{\alpha\beta}$ is the second fundamental form of $\partial \Omega$ at $x_0$. Differentiating the
boundary condition $u = \varphi$ on $\partial \Omega$ twice, we can find a constant $C$ depending on
$|\varphi|_{C^2 (\partial \Omega)}$ and $|u|_{C^1 (\ol \Omega)}$ such that
\begin{equation}
\label{BC2-2}
|u_{\alpha \beta} (0)| \leq C \mbox{  for } \alpha, \beta \leq n - 1.
\end{equation}
Next, we establish the estimate
\begin{equation}
\label{BC2-3}
|u_{\alpha n} (0)| \leq C \mbox{  for } \alpha \leq n - 1.
\end{equation}
For $x\in\partial\Omega$ near the origin, let
\[
T_\alpha = \partial_{\alpha}
  + \sum_{\beta < n} B_{\alpha \beta} (x_{\beta}\partial_{n}-x_{n}\partial_{\beta}), \;\;
  \mbox{ for } \alpha < n,
\]
$T_n = \partial_n$ and $\omega_\delta = \{x \in \Omega: \rho (x') < x_n < \rho (x') + \delta^2 , |x'| < \delta\}$.
We have
\[
S_k^{ij} (T_\alpha u)_{ij} = T_\alpha f.
\]
It follows that
\begin{equation}
\label{bd-2}
|S_k^{ij} (T_\alpha (u - \varphi))_{ij}| \leq C \Big(S_{k - 1} + f^{1 - 1/ (k - 1)}\Big)
\end{equation}
and
\begin{equation}
\label{bd-3}
|T_\alpha (u - \varphi)| \leq C |x'|^2 \mbox{ on } \partial \Omega \cap \overline{\omega}_\delta
    \mbox{ for } \alpha < n
\end{equation}
when $\delta$ is sufficiently small since $u = \varphi$ on $\partial \Omega$. Because $\Omega$ is
uniformly $(k - 1)$- convex, there exist positive constants $\theta$ and $K$ such that
\[
(\kappa_1 - 2 \theta, \ldots, \kappa_{n-1} - 2 \theta, 2 K) \in \Gamma_k.
\]
Define
\begin{equation}
\label{BC2-6}
\Psi = \rho (x') - x_n - \theta |x'|^2 + K x_n^2.
\end{equation}
Note that the boundary $\partial \omega_\delta$ consists three parts: $\partial \omega_\delta
= \partial_1 \omega_\delta \cup \partial_2 \omega_\delta \cup \partial_3 \omega_\delta$, where
$\partial_1 \omega_\delta$, $\partial_2 \omega_\delta$ are  defined by $\{x_n=\rho\} \cap\overline{\omega}_{\delta}$, $\{ x_n=\rho+\delta^2\}\cap\overline{\omega}_{\delta}$
respectively, and $\partial_3 \omega_\delta$ is defined by $\{|x'| = \delta\}\cap\overline{\omega}_{\delta}$.

We see that when $\delta$ is sufficiently small (depending on $\theta$ and $K$),
$\Psi \leq 0$ in $\omega_\delta$ and furthermore,
\begin{equation}
\label{BC2-12}
\begin{aligned}
\Psi \leq & - \frac{\theta}{2} |x'|^2, \mbox{ on } \partial_1 \omega_\delta\\
\Psi \leq & - \frac{\delta^2}{2}, \mbox{ on } \partial_2 \omega_\delta\\
\Psi \leq & - \frac{\theta \delta^2}{2}, \mbox{ on } \partial_3 \omega_\delta.
\end{aligned}
\end{equation}
We derive from \eqref{BC2-1} that $\lambda (D^2 \Psi) \in \Gamma_k$ and
\begin{equation}
\label{BC2-4}
S_k^{ij} \Psi_{ij} \geq \eta_0 S_{k - 1} \mbox{ on } \overline{\omega}_\delta
\end{equation}
for some uniform constant $\eta_0 > 0$ by further requiring $\delta$ small.

We claim that there exist uniform positive constants $A$ and $\delta$ such
that $A \Psi \pm T_\alpha (u - \varphi) \leq 0$ on $\overline{\omega}_\delta$. It is easy to
get \eqref{BC2-3} from the claim since $A \Psi (0) \pm T_\alpha (u - \varphi) (0) = 0$.

Now we prove the claim. We first note that $A \Psi \pm T_\alpha (u - \varphi) \leq 0$
on $\partial \omega_\delta$ by \eqref{BC2-12} when $A$ is sufficiently large.
So we may suppose
\[
W := \max_{\overline{\omega}_\delta} (A \Psi
\pm T_\alpha (u - \varphi))
\]
is attained at an interior point $x_0 \in \omega_\delta$
and $A$ is sufficiently large to be chosen. We may assume, at $x_0$, $A \Psi
\pm T_\alpha (u - \varphi) \geq 0$ for otherwise we are done. Now we consider two cases:
(i) $\Delta u (x_0) \leq 1$ and (ii) $\Delta u (x_0) > 1$.

\textbf{Case (i).} As in the gradient estimates, we see $|u_{ij} (x_0)| \leq C_0 \Delta u (x_0) \leq C_0$
for some uniform positive constant $C_0$ and each $1 \leq i, j \leq n$. We note that, at $x_0$,
\[
\begin{aligned}
0 = \,& (A \Psi \pm T_\alpha (u - \varphi))_n\\
  = \,& - A + 2 A K x_n \nonumber\\
      & \pm \left[(u - \varphi)_{\alpha n}+\sum_{\beta<n}B_{\alpha\beta}(x_{\beta}(u-\varphi)_{nn}-x_n(u-\varphi)_{\beta n})-\sum_{\beta<n}B_{\alpha\beta}(u-\varphi)_{\beta}\right]\\
  < \,& 0,
\end{aligned}
\]
if we choose $\delta$ is sufficiently small and $A$ is sufficiently large using the bound of $|D^2u|$ and $|Du|$ at $x_0$. Then we have a contradiction.

\textbf{Case (ii).} By \eqref{GC2-9} and \eqref{bd-2}, we see at $x_0$ where $W$ is attained,
\[
\begin{aligned}
0 \geq \,& S_k^{ij} (A \Psi \pm T_\alpha (u - \varphi))_{ij}
  \geq A \eta_0 S_{k - 1} - C \Big(S_{k - 1} + f^{1 - 1/ (k - 1)}\Big)\\
  \geq \,& \frac{A}{2} \eta_0 S_{k-1} + \frac{A}{2} \eta_0 c_0 f^{1 - 1/ (k - 1)}
     - C \Big(S_{k - 1} + f^{1 - 1/ (k - 1)}\Big) > 0
\end{aligned}
\]
provided $A$ is sufficiently large, which is a contradiction.
The claim follows and \eqref{BC2-3} is proved.

For the homogenous problem, combining \eqref{BC2-2},\eqref{BC2-3} with the estimates for double normal derivatives in \cite{Dong06}, the conditions \eqref{C1-8} and \eqref{GC2-10} are sufficient to obtain the existence of $C^{1,1}$ solutions, which is a slight different from Dong's theorem \cite{Dong06}.
\begin{theorem}
\label{main-result-2}
Suppose $\Omega$ is uniformly $(k - 1)$-convex with $\partial \Omega \in C^{3,1}$, $f \geq 0$ and
 there exists a positive constant $A$ such that \eqref{C1-8} and \eqref{GC2-10} hold.
Then the equation \eqref{1-1} with $\varphi \equiv 0$ has a unique $k$-convex solution $u \in C^{1,1} (\ol \Omega)$
satisfying the estimates \eqref{main-estimate}, where the constant $C$ depends on $\Omega$, $A$, $(\sup_\Omega f)^{-1}$
and modulus continuity of $f$ in $\Omega$.
\end{theorem}

\section{An inequality on the boundary}

Suppose  $W$  is a $(0,2)$ tensor on $\overline{\Omega}$, namely $W\in C^2(T^*\overline{\Omega}\otimes T^*\overline{\Omega})$,
where $T^*\overline{\Omega}$ is the co-tangent bundle of $\overline{\Omega}$.
Let $W'$
be the projection of $W|_{\partial\Omega}$ in the bundle $T^* \partial \Omega\otimes T^*\partial\Omega$, where $T^*\partial \Omega$ is the co-tangent bundle of $\partial \Omega$.  $\lambda' (W')$ denotes
the eigenvalue vector of $W'$ with respect to the induced metric on $\partial \Omega$.
In this section, we prove
\begin{lemma}
\label{cor-lemma}
There exists a uniform positive constant $\delta_0$ such that
\begin{equation}
\label{cor-1}
\sigma_{k-1} (\lambda' [(D^2 u)']) \geq \delta_0 f
\end{equation}
on $\partial \Omega$.
\end{lemma}
\begin{proof}
We use an idea being closed to \cite{T}.
Define $G (r) = \sigma_{k-1}^{1/(k-1)} (\lambda' (r))$ for a $(n-1) \times (n-1)$ matrix
$r$ with its eigenvalues $\lambda' (r) \in \Gamma_{k-1}$. Let
\[
m := \inf_{\partial \Omega} \frac{G ((D^2 u)')}{\widetilde{f}},
\]
where $\widetilde{f} = f^{1/(k-1)}$. It suffices
to prove $m \geq \delta_0$ for some positive constant $\delta_0$.
Suppose $m$ is attained at $x_0 \in \partial \Omega$. We may assume that $x_0$ is the origin and the
positive $x_n$-axis is in the interior normal direction to $\partial \Omega$ at the origin as before.
Choose local orthonormal frames
$\{e_1, \ldots, e_n\}$ around $x_0$ such that $e_n$ is the interior normal to $\partial \Omega$.
$\nabla$ denotes the standard connection of $\mathbb{R}^n$. Write $e_i (x) = e^j_i (x)\partial_j$ for $i = 1, \ldots n$, where $\partial_1,\cdots,\partial_n$ is the  rectangular coordinate system. Thus, we have
 \[
\nabla_i u:=\nabla_{e_i}u= e_i^j \partial_ju=e_i^ju_j
\]
and
\[
\nabla_{ij} u:=\nabla_{e_i}\nabla_{e_j}u =e_i^ke_j^l\partial_{k}\partial_l u = e_i^k e_j^l u_{kl}.
\]
We may also assume that $e_i^j (x_0) = \delta_{ij}$
for $1 \leq i, j \leq n$ and $\{\nabla_{\alpha \beta} u (x_0)\}_{1 \leq \alpha, \beta \leq n-1}$ is diagonal.
Let $\ul u$ be a $k$-convex subsolution satisfying \eqref{2-1}.
Since $u - \ul u = 0$ on $\partial \Omega$, we find
\begin{equation}
\label{cor-2}
\nabla_{\alpha \beta} u = \nabla_{\alpha \beta} \ul u - \nabla_n (u - \ul u) \sigma_{\alpha \beta}, \ 1 \leq \alpha, \beta \leq n - 1
\end{equation}
on $\partial \Omega$ near $x_0$, where $\sigma_{\alpha \beta} = \langle D_{e_\alpha} e_\beta, e_n\rangle$ is the second fundamental form of $\partial \Omega$.
Let
\[
G^{\alpha \beta}_0 = \frac{\partial G}{\partial r_{\alpha \beta}} (\nabla_{\alpha \beta} u (x_0)), 1 \leq \alpha, \beta \leq n-1.
\]
By the concavity of $G$, we have
\begin{equation}
\label{cor-3}
G^{\alpha \beta}_0 (r_{\alpha \beta} - \nabla_{\alpha \beta} u (x_0))
  \geq G (r) - G (\nabla_{\alpha \beta} u (x_0))
\end{equation}
for any matrix $r$ satisfying $\lambda' (r) \in \Gamma_{k-1}$. It follows from \eqref{cor-2}, \eqref{cor-3}
and the definition of $m$ that
\begin{equation}
\label{cor-4}
\begin{aligned}
- G^{\alpha \beta}_0 \sigma_{\alpha \beta} \nabla_n (u - \ul u)
   = \,& G^{\alpha \beta}_0 (\nabla_{\alpha \beta} u - \nabla_{\alpha \beta} \ul u)\\
   = \,& G^{\alpha \beta}_0 (\nabla_{\alpha \beta} u - \nabla_{\alpha \beta} u (x_0))\\ & + G^{\alpha \beta}_0 \nabla_{\alpha \beta} u (x_0)
      - G^{\alpha \beta}_0 \nabla_{\alpha \beta} \ul u\\
\geq \,& G (\nabla_{\alpha \beta} u) - G (\nabla_{\alpha \beta} u (x_0))\\ & + G^{\alpha \beta}_0 \nabla_{\alpha \beta} u (x_0)
      - G^{\alpha \beta}_0 \nabla_{\alpha \beta} \ul u\\
\geq \,& m \widetilde{f} 
     - G^{\alpha \beta}_0 \nabla_{\alpha \beta} \ul u
\end{aligned}
\end{equation}
on $\partial \Omega$ near $x_0$. Since $\Omega$ is uniformly $(k-1)$-convex and $\sigma_{\alpha \beta}$ is the second
fundamental form of $\partial \Omega$, we have
\[
\lambda' (\{\sigma_{\alpha \beta} - \theta \delta_{\alpha \beta}\}) \in \Gamma_{k-1}
\]
on $\partial \Omega$ near $x_0$ for some positive constant $\theta$ depending only on the geometry of $\partial \Omega$.
Thus, by the concavity of $G$, we get
\begin{equation}
\label{useful}
\begin{aligned}
G^{\alpha \beta}_0 \sigma_{\alpha \beta} \geq \,& G (\sigma_{\alpha \beta} - \theta \delta_{\alpha \beta}) + \theta \sum_{\alpha =1}^{n-1} G^{\alpha\alpha}_0\\
   \geq \,& \gamma + \theta \sum_{\alpha =1}^{n-1} G^{\alpha\alpha}_0
\end{aligned}
\end{equation}
for some positive constant $\gamma$.
Let
\[
\Omega_\delta := \{x \in \Omega: |x - x_0| < \delta\}
\]
for $\delta$ sufficiently small and
$\eta := G^{\alpha \beta}_0 \sigma_{\alpha \beta}$. Define the following barrier in $\Omega_\delta$,
\[
\tilde{\Phi} = - \nabla_n(u - \ul u) + \frac{1}{\eta} \Big(
      G^{\alpha \beta}_0 \nabla_{\alpha \beta} \ul u - m \widetilde{f} \Big).
\]
It follows that $\tilde{\Phi} \geq 0$ on $\partial \Omega \cap \partial \Omega_\delta$ and $\tilde{\Phi} (x_0) = 0$.
Since $e_j^i = \delta_{ij}$ at $x_0 = 0$, we have
\begin{equation}
\label{boundary1}
|e_\alpha^\alpha (u - \ul u)_\alpha| \equiv |\sum_{j\neq \alpha} e_\alpha^j (u - \ul u)_j| \leq C\sum_{j\neq \alpha}|e_{\alpha}^j|\leq C_0 |x|
\end{equation}
on $\partial \Omega \cap \partial \Omega_\delta$ for $1 \leq \alpha \leq n-1$ and the constant $C_0$ depends only on
the bound of $|D (u - \ul u)|$ and $|D e_\alpha^j|$ for $j \neq \alpha$. We may also assume that $\inf_{\Omega_\delta} e_\alpha^\alpha e_n^n \geq c_0 > 0$
for any $1 \leq \alpha \leq n-1$ and some positive constant $c_0$ by choosing $\delta$ sufficiently small.
By \eqref{boundary1} and that $e_j^i = \delta_{ij}$ at $x_0 = 0$ again, we have
\[
\begin{aligned}
\nabla_n (u - \ul u) = \,& e_n^n (u - \ul u)_n + \sum_{\alpha = 1}^{n-1} e_n^{\alpha} (u - \ul u)_\alpha\\
  = \,& e_n^n (u - \ul u)_n + \sum_{\alpha = 1}^{n-1} \frac{e_n^{\alpha} e^\alpha_\alpha (u - \ul u)_\alpha}{e^\alpha_\alpha}
  \geq e_n^n (u - \ul u)_n - CC_0 |x|^2
\end{aligned}
\]
on $\partial \Omega \cap \partial \Omega_\delta$, where the constant $C$ depends only on $\inf_{\Omega_\delta} e_\alpha^{\alpha}$
and the bound of $|D e_n^\alpha|$, $1 \leq \alpha \leq n-1$.
Since $\tilde{\Phi} \geq 0$ on $\partial \Omega \cap \partial \Omega_\delta$ and $\tilde{\Phi} (x_0) = 0$, we have
\[
\Phi := - (u - \ul u)_n + \frac{1}{\eta e_n^n} \Big(
      G^{\alpha \beta}_0 \nabla_{\alpha \beta} \ul u - m \widetilde{f} \Big) + M |x|^2 \geq 0
\]
on $\partial \Omega \cap \partial \Omega_\delta$ and $\Phi (x_0) = 0$,
where $M = CC_0/c_0$.
Let $\Psi$ be the function defined in \eqref{BC2-6}. Note that $\Psi$ satisfies \eqref{BC2-12}
and \eqref{BC2-4}. Consider the function $B \Psi - \Phi$ on $\overline{\omega}_\delta$, where
$B$ is a positive sufficiently large constant to be chosen later.

Because of \eqref{bd-2} and $\widetilde{f} \in C^{1,1} (\ol \Omega)$, we have
\[
S^{ij}_k \Phi_{ij} \leq C \Big(S_{k - 1} + f^{1 - 1/ (k - 1)}\Big) \mbox{ on } \overline{\omega}_\delta.
\]
As in Section 4, we can take sufficiently large $B$ to staisfy $B \Psi - \Phi \leq 0$ on $\partial \omega_\delta$.
Suppose the maximum of $B \Psi - \Phi$ on $\overline{\omega}_\delta$ is achieved at an interior point $x_1 \in \omega_\delta$.
We see, at $x_1$,
\[
(B \Psi - \Phi)_n = - B + 2 BK x_n - \Phi_n = 0,
\]
which implies that $u_{nn} (x_1) \geq B/2$  provided $B$ is sufficiently large
and $\delta$ is sufficiently small. Therefore, we get $\Delta u (x_1) \geq B/2$.
Using \eqref{GC2-9} again, we have,
at $x_1$,
\[
0 \geq S^{ij}_k (B \Psi - \Phi)_{ij} > 0
\]
if $B$ is sufficiently large, which is a contradiction. Therefore, $B \Psi - \Phi \leq 0$ on $\overline{\omega}_\delta$.
It follows that
$\Phi_n (x_0) \geq B \Psi_n (x_0) \geq -C$.
Thus, we get
\[
u_{nn} (x_0) \leq C (1 + m),
\]
where $C$ depends on $|u|_{C^1 (\overline{\Omega})}$, $\sup_{\partial \Omega} |(D^2 u)'|$, $|\ul u|_{C^2 (\overline{\Omega})}$
and $|\widetilde{f}|_{C^{1,1} (\overline{\Omega})}$. Without loss of generality, we may assume $m \leq 1$. Combining with $\Delta u\geq 0$, we then obtain a bound
\begin{equation}
\label{B-3}
|u_{nn} (x_0)| \leq C.
\end{equation}
Recall that $\nabla_{\alpha \beta} u (x_0) = u_{\alpha \alpha} (x_0) \delta_{\alpha \beta}$ for $1 \leq \alpha, \beta \leq n-1$.
Let $b_\alpha = u_{\alpha \alpha} (x_0)$ for $1 \leq \alpha \leq n-1$ and
$b = (b_1, \ldots, b_{n-1})$.
By the equation \eqref{1-1}, we see, at $x_0$,
\begin{equation}
\label{fromeqn}
u_{nn} \sigma_{k - 1} (b) - \sum_{\alpha = 1}^{n-1} u_{\alpha n}^2 \sigma_{k - 2; \alpha} (b) + \sigma_k (b) = f.
\end{equation}
Since we have
\[
\begin{aligned}
\frac{\sigma_k (b)}{\sigma_{k-1} (b)} = \,& \frac{\sum_{\alpha = 1}^{n - 1} b_\alpha \sigma_{k - 1; \alpha} (b)}{k \sigma_{k-1} (b)}\\
        = \,& \frac{\sum_{\alpha = 1}^{n - 1} b_\alpha (\sigma_{k-1} (b) - b_\alpha \sigma_{k - 2; \alpha} (b))}{k \sigma_{k-1} (b)}\\
        = \,& \frac{\sum_{\alpha = 1}^{n - 1} b_\alpha}{k}
           - \frac{\sum_{\alpha = 1}^{n - 1} b_\alpha^2 \sigma_{k - 2; \alpha} (b)}{k \sigma_{k-1} (b)}
        \leq \frac{\sum_{\alpha = 1}^{n - 1} b_\alpha}{k} \leq C,
\end{aligned}
\]
then we get
\[
\frac{f (x_0)}{\sigma_{k - 1} (b)} \leq u_{nn} (x_0) + \frac{\sigma_k (b)}{\sigma_{k-1} (b)} \leq C.
\]
We are done.
\end{proof}

\section{Estimates for convex solutions, the case $k=2$}

From now on, we assume that $u \in C^3 (\Omega) \cap C^2 (\ol \Omega)$ is a convex
solution of \eqref{1-1} and that $\Omega$ is uniformly convex. We consider an arbitrary point $x_0 \in \partial \Omega$.
As in Section 4, we assume $x_0$ is the origin, the
positive $x_n$-axis is in the interior normal direction to $\partial \Omega$ at the origin,
and near the origin, the boundary $\partial \Omega$ is given by \eqref{BC2-1}.

In this and the following sections we always use the notation $x' = (x_1, \ldots, x_{n-1}) \in \mathbb{R}^{n-1}$ to denote the
projection of $x = (x_1, \ldots, x_n) \in \mathbb{R}^n$ in $\mathbb{R}^{n-1}$ by its first $n-1$ coordinates.

We may assume $\{u_{\alpha \beta} (0)\}_{1 \leq \alpha, \beta \leq n-1}$ is diagonal and
\[
\{u_{\alpha \beta} (0)\}_{1 \leq \alpha, \beta \leq n-1} = \mathrm{diag} \{b_1, \ldots, b_{n-1}\}.
\]
We may further assume $b_1 \leq \cdots \leq b_{n-1}$.

Let  $\nabla_{\zeta\eta}$ denote the second order covariant derivative
with respect to the standard connection $\nabla$ for any local vector fields $\zeta$, $\eta$ in $\Omega$ near the origin.
In this and the next section, we shall prove
\begin{lemma}
\label{lemma-2}
Let $\nu$ denote the unit interior normal to $\partial \Omega$ and $\tau_\beta = \partial_\beta + \rho_\beta \partial_n$ be the tangential vector field of $\partial \Omega$ for $\beta \leq n-1$ near the origin.
For $\alpha = n - (k-1), \ldots, n-1$, there exist two positive constants $\theta_\alpha$ and $C$ depending only on $\Omega$,
$|\varphi|_{C^{3,1} (\ol \Omega)}$ and $|f^{1/(k-1)}|_{C^{1,1} (\ol \Omega)}$ such that if any point $x = (x_1, \ldots, x_n) \in \partial \Omega$
satisfies $|x_\beta| \leq \theta_\alpha \sqrt{b_\alpha}$ for any $1 \leq \beta \leq \alpha$
and $|x_\beta| \leq \theta_\alpha \frac{b_\alpha}{\sqrt{b_\beta}}$ for any $\alpha + 1 \leq \beta \leq n-1$, we have the estimates
\begin{equation}
\label{jw-18}
|\nabla_{\nu \tau_\beta} u (x)| \leq C \sqrt{b_\alpha}
\end{equation}
for all $\beta = 1, \ldots, \alpha$.
\end{lemma}
Let $\omega_1 := \{x \in \Omega: x_n < \epsilon_1 b_{n-1}\}$ with a positive constant $\epsilon_1$ sufficiently small to be chosen later. In this section, we prove
\begin{equation}
\label{jw-12}
|\nabla_{\nu \xi}u| \leq C \sqrt{b_{n-1}} \mbox{ on } \partial \omega_1 \cap \partial \Omega,
\end{equation}
for any convex solution $u$ of \eqref{1-1}, where $\xi$ is any unit tangent vector of $\partial \Omega$.
Note that \eqref{jw-12} implies Lemma \ref{lemma-2} for $\alpha = n-1$ and that Lemma \ref{lemma-2} follows
immediately for the case $k=2$.

Since $f^{1/(k-1)} \in C^{1,1} (\ol \Omega)$, we have
\begin{equation}
\label{jw-1}
|\tilde{f}_\alpha (0)| \leq C \sqrt{\tilde{f} (0)}, \mbox{ for } \alpha \leq n-1,
\end{equation}
where $\tilde{f} := f^{1/(k-1)}$ and the constant $C$ depends only on $|\tilde{f}|_{C^{1,1} (\ol \Omega)}$.
By \eqref{cor-1}, \eqref{jw-1} and Taylor's expansion of $\tilde{f}$, we have
\[
\tilde{f} (x) \leq \tilde{f} (0) + \sum_{i=1}^n \tilde{f}_i (0) x_i + C |x|^2
   \leq C \left(\sigma_{k-1}^{1/(k-1)} (b) + \sigma_{k-1}^{1/2(k-1)} (b) |x'| + x_n + |x|^2\right)
\]
in $\Omega$ near the origin, where $x'=(x_1,\cdots, x_{n-1})$. It follows that
\begin{equation}
\label{taylor}
f (x) \leq C \left(\sigma_{k-1} (b) + \sqrt{\sigma_{k-1} (b)} |x'|^{k-1} + x_n^{k-1} + |x|^{2 (k-1)}\right)
\end{equation}
in $\Omega$ near the origin.
Define
$\omega := \{x \in \Omega: x_n < 2 b_{n-1}\}$. Therefore, we have
\begin{eqnarray}\label{add1}
|x'|\leq C\sqrt{b_{n-1}}, \text{ for } x\in\overline{\omega}
\end{eqnarray}
by \eqref{BC2-1}.
In the following, we always suppose $b_{n-1}$ is sufficiently small, otherwise by \eqref{BC2-3}, we have got \eqref{jw-12}.
Thus, by \eqref{BC2-1} and \eqref{taylor} we get
\begin{equation}
\label{taylor-1}
f \leq C b_{n-1}^{k-1} \mbox{ in } \omega.
\end{equation}
Subtracting $D u (0) \cdot x + u (0)$, we may assume $\inf_\Omega u = u (0) = 0$ and $D u (0) = 0$.
By Taylor's expansion, we have
\begin{equation}
\label{jw-2}
\varphi (x', \rho (x')) = \frac{1}{2} \sum_{\alpha \leq n-1} b_\alpha x_\alpha^2 +
  \frac{1}{6} \sum_{\alpha, \beta, \gamma \leq n-1} \varphi_{\alpha \beta \gamma} (0) x_\alpha x_\beta x_\gamma
    + O (|x'|^4),
\end{equation}
near the origin, where $\varphi_{\alpha \beta \gamma} (x') = \frac{\partial^3 \varphi (x', \rho (x'))}{\partial x_\alpha \partial x_\beta \partial x_\gamma}$.
Since $u = \varphi \geq 0$ on $\partial \Omega$ and \eqref{add1}, we find
\begin{equation}
\label{jw-3}
|\varphi_{\alpha \beta \gamma} (0)| \leq C \sqrt{b_{n-1}}.
\end{equation}
We extend the boundary function $\varphi$ to be the right hand side of \eqref{jw-2} near the origin in $\Omega$ and still denote it by $\varphi$.
By the convexity of $u$, the maximum value  of $u$ on $ \{x_n = 2 b_{n-1}\}$ achieves on $\partial\Omega$.
Therefore, again by convexity, we have
\begin{equation}
\label{jw-4}
|u - \varphi| \leq \sup_{\partial\Omega\cap\partial \overline{\omega}}(|u|+|\varphi|)\leq C b_{n-1}^2 \mbox{ on } \ol \omega.
\end{equation}
Let $d (x) := \mathrm{dist} (x, \partial \Omega)$ and
\[
v := - d + \frac{1}{8 b_{n-1}} d^2.
\]
We have $d \leq x_n$ near the origin so that $v \leq 0$ in $\omega$. Moreover, at $x=0$, we have
\begin{eqnarray}\label{d2v}
D^2v\geq \mathrm{diag}\left\{\frac{\kappa_1}{2},\cdots,\frac{\kappa_{n-1}}{2},\frac{1}{4b_{n-1}}\right\},
\end{eqnarray}
where $\kappa_1,\cdots, \kappa_{n-1}$ are the principal curvatures of $\partial \Omega$ at origin.
Thus  $v$ is convex near origin and
\[
\sigma_k (D^2 v) \geq \frac{\delta_0}{b_{n-1}} \mbox{ in } \omega
\]
for some positive constant $\delta_0$ because $\Omega$ is uniformly convex.
Let
\[
F (r) = \sigma_k^{1/k} (\lambda(r)) \mbox{ with } \lambda(r) \in \Gamma_k
\]
where $r$ is a symmetric matrix  and $\lambda(r)$ is the eigenvalue vector of $\lambda$,
and
\[
F^{ij} = \frac{\partial F(D^2 u)}{\partial u_{ij}}.
\]
By using \eqref{d2v}, $D^2v-\delta_1I$ is a positive definite matrix for some sufficiently small constant $\delta_1$, where $I$ is the identical matrix. By the concavity of $F$ and the uniformly convexity of $\Omega$, we find
\begin{equation}
\label{jw-5}
F^{ij} v_{ij} \geq \delta_1 \left(\frac{1}{b_{n-1}^{1/k}} + \sum_{i=1}^n F^{ii}\right)
  \mbox{ in } \omega.
\end{equation}
Let $\omega_0 := \{x \in \Omega: x_n < \epsilon b_{n-1}\}$, where $\epsilon$ is a positive constant sufficiently small to be determined.
For any fixed $y = (y_1, \cdots, y_n)
\in \partial \omega_0 \cap \partial \Omega$, we have
\begin{equation}
\label{add-2}
x_n - y_n \geq (2 - \epsilon) b_{n-1} \mbox{ on } \partial \omega \cap \{x_n = 2 b_{n-1}\}.
\end{equation}
Next, for any $x = (x_1, \cdots, x_n) \in \partial \omega \cap \partial \Omega$, we have
\[
\begin{aligned}
x_n - y_n = \,& \sum_{\beta = 1}^{n-1}\rho_\beta (y') (x_\beta - y_\beta) + O (|x'-y'|^2)\\
   \geq \,& \sum_{\beta = 1}^{n-1}\rho_\beta (y') (x_\beta - y_\beta) - \kappa |x'-y'|^2
\end{aligned}
\]
by Taylor's expansion of $\rho$ at $y'$, where $\kappa$ is a positive constant depending only on the principal curvatures of $\partial \Omega$.
Let
\[
L (x) := \sum_{\beta = 1}^{n-1}\rho_\beta (y') (x_\beta - y_\beta).
\]
It follows that
\begin{equation}
\label{add-1}
w (x) := x_n - y_n - L (x) + \kappa |x'-y'|^2 \geq 0 \mbox{ on } \partial \omega \cap \partial \Omega.
\end{equation}
Note that $\rho_\beta (0) = 0$ for each $1 \leq \beta \leq n-1$. We find, by \eqref{BC2-1},
\[
|L (x)| \leq |D \rho (y')| \cdot |x'-y'| \leq C |y'| \cdot |x'-y'| \leq C \sqrt{\epsilon} b_{n-1}
\]
for any $x \in \ol \omega$. In particular, by \eqref{add-2}, we have
\begin{equation}
\label{add-3}
w (x) \geq (2 - \epsilon - C \sqrt{\epsilon}) b_{n-1} \geq b_{n-1}
\end{equation}
on $\partial \omega \cap \{x_n = 2 b_{n-1}\}$ by fixing $\epsilon$ small enough.
Thus, by \eqref{jw-4}, \eqref{add-1} and \eqref{add-3}, we have
\begin{equation}
\label{jw-17}
|u - \varphi| \leq C b_{n-1} w \mbox{ on } \partial \omega.
\end{equation}
From \eqref{taylor-1}, \eqref{jw-17}, \eqref{jw-5} and using Lemma \ref{cor-lemma}, we derive
\begin{equation}
\label{jw-7}
\begin{aligned}
F^{ij} (b_{n-1} (A_1 v - A_2 w) \pm (u - \varphi))_{ij} \geq \,& 0 \mbox{ in } \omega\\
  b_{n-1} (A_1 v - A_2 w) \pm (u - \varphi) \leq \,& 0 \mbox{ on } \partial \omega
\end{aligned}
\end{equation}
by choosing $A_1 \gg A_2 \gg 1$. It follows from the maximal principle that
\begin{equation}
\label{plus2}
|u_{\nu} (y)| \leq C b_{n-1},
\end{equation}
for $y\in \partial\Omega\cap\partial \omega_0$.
By \eqref{BC2-1}, \eqref{jw-2}, \eqref{plus2} and the convexity of $u$, we have
\begin{equation}
\label{plus1}
\sup_{\ol \omega_0} |u_\alpha| \leq \sup_{\partial \omega_0 \cap \partial \Omega} |\varphi_\alpha - u_n \rho_\alpha| \leq C b_{n-1}^{3/2}, \mbox{ for each } 1 \leq \alpha \leq n-1.
\end{equation}
Let  $\epsilon_1 = \epsilon/2$. For any $x = (x', \epsilon_1 b_{n-1}) \in \partial \omega_1 \cap
\{x_n = \epsilon_1 b_{n-1}\}$, by the convexity of $u$, \eqref{jw-4} and \eqref{plus1},
\begin{equation}
\label{add-6}
- \epsilon_1 b_{n-1} u_n (x) \leq u(0) - u(x) + \sum_{\alpha = 1}^{n-1} x_\alpha u_\alpha (x) \leq C b_{n-1}^2.
\end{equation}
It follows that
\begin{equation}
\label{plus3}
u_n (x) \geq - \frac{C}{\epsilon_1} b_{n-1}.
\end{equation}
On the other hand, we fix a point $y = (y', \epsilon b_{n-1}) \in \partial \omega_0 \cap \partial \Omega$.
By the convexity of $u$, \eqref{jw-4} and \eqref{plus1}, we find
\begin{equation}
\label{add-7}
\epsilon_1 b_{n-1} u_n (x) \leq u(y) - u(x) + \sum_{\alpha = 1}^{n-1} (y_\alpha - x_\alpha) u_\alpha (x) \leq C b_{n-1}^2
\end{equation}
and we obtain
\begin{equation}
\label{plus4}
u_n (x) \leq \frac{C}{\epsilon_1} b_{n-1}.
\end{equation}
Combining \eqref{plus1}, \eqref{plus3} with \eqref{plus4}, we get
\[
|Du| \leq C b_{n-1} \mbox{ on } \partial \omega_1.
\]
By the convexity of $u$ again, we have
\begin{equation}
\label{jw-8}
|Du| \leq C b_{n-1} \mbox{ in } \omega_1.
\end{equation}
In the following, we always denote $\nabla_{\alpha}=\nabla_{\tau_{\alpha}}$.
Note that $|\rho_\alpha| \leq C |x'| \leq C b_{n-1}^{1/2}$ on $\ol \omega_1$.
By \eqref{plus1} and \eqref{jw-8}, we see
$|\nabla_\alpha u| \leq C b_{n-1}^{3/2}$ on $\overline{\omega}_1$. It follows that
\begin{equation}
\label{jw-6}
|\nabla_\alpha (u - \varphi)| \leq C \sqrt{b_{n-1}} w \mbox{ on } \partial \omega_1
\end{equation}
by using \eqref{add-3}.
Since $\tau_\alpha$ is a tangential vector field of $\partial \Omega$ near the origin, we have
\begin{equation}
\label{smooth}
|\nabla_\alpha \tilde{f}| \leq C\sqrt{\tilde{f}}.
\end{equation}
near the origin. The proof of \eqref{smooth} can be found in Lemma 3.1 of \cite{Blocki03}.
Differentiating the equation $F (D^2 u) = f^{1/k}$ and using Lemma \ref{cor-lemma}, we get
\begin{equation}
\label{jw-10}
\begin{aligned}
|F^{ij} (\nabla_\alpha u)_{ij}| \leq \,& |\nabla_\alpha f^{1/k}| + C \sum_{i,j=1}^n F^{ij} u_{ij}
   + C |Du| \sum_{i=1}^n F^{ii}\\
    \leq \,& Cf^{\frac{1}{k}-\frac{1}{2(k-1)}} + C b_{n-1} \sum_{i=1}^n F^{ii}\\
    \leq\, &C b_{n-1}^{1/2-1/k}+ C b_{n-1} \sum_{i=1}^n F^{ii}
    \end{aligned}
\end{equation}
in $\omega_1$. By \eqref{jw-3}, we have
\begin{equation}
\label{jw-11}
|F^{ij} (\nabla_\alpha \varphi)_{ij}| \leq C \sqrt{b_{n-1}} \sum_{i=1}^n F^{ii} \mbox{ in } \omega_1.
\end{equation}
Combining \eqref{jw-6}, \eqref{jw-10}, \eqref{jw-11} and \eqref{jw-5}, we can choose positive constants
$A_1 \gg A_2 \gg 1$ again such that
\[
\begin{aligned}
F^{ij} \left(\sqrt{b_{n-1}} (A_1 v - A_2 w) \pm \nabla_\alpha (u - \varphi)\right)_{ij} \geq \,& 0 \mbox{ in } \omega_1\\
  A \sqrt{b_{n-1}} (A_1 v - A_2 w) \pm \nabla_\alpha (u - \varphi) \leq \,& 0 \mbox{ on } \partial \omega_1
\end{aligned}
\]
By the maximal principle again we get
\[
\sqrt{b_{n-1}} (A_1 v - A_2 w) \pm \nabla_\alpha (u - \varphi) \leq 0 \mbox{ in } \omega_1
\]
and
\begin{equation}
\label{jw-13}
|\nu(\nabla_{\alpha}u) (y)| \leq C \sqrt{b_{n-1}} \mbox{ for } \alpha \leq n-1,
\end{equation}
for any $y\in\partial\Omega\cap\partial\omega_1$, where $\nu$ is the unit interior normal of $\partial \Omega$ at $y$. Therefore, in view of \eqref{jw-8}, we get \eqref{jw-12}.

Now we consider the case $k = 2$.
We have, at the origin,
\[
u_{nn} \sigma_1 (b) - \sum_{\alpha \leq n-1} u_{n \alpha}^2 + \sigma_2 (b) = f.
\]
Thus, by \eqref{cor-1} and \eqref{jw-12}, we obtain
\begin{equation}
\label{jw-14}
u_{nn} (0) \leq C.
\end{equation}

\section{Estimates for convex solutions, the case $k \geq 3$}

In this section, we continue to prove Lemma \ref{lemma-2} for the case $k \geq 3$.

As \cite{GTW99}, we will prove Lemma \ref{lemma-2} by induction. Our induction hypothesis is that for some given $\alpha = n - (k-1), \ldots, n-2$ and any index
$\alpha + 1\leq \gamma\leq  n-1$, there exist two positive constants $\theta_\gamma$ and $C$ depending only on $\Omega$,
$|\varphi|_{C^{3,1} (\ol \Omega)}$ and $|\tilde{f}|_{C^{1,1} (\ol \Omega)}$ such that if any point $(x', \rho (x')) \in \partial \Omega$
satisfies $|x_\beta| \leq \theta_\gamma \sqrt{b_\gamma}$ for any $1 \leq \beta \leq \gamma$
and $|x_\beta| \leq \theta_\gamma \frac{b_\gamma}{\sqrt{b_\beta}}$ for any $\gamma + 1 \leq \beta \leq n-1$, we have the estimates
\begin{equation}
\label{jw-19}
|\nabla_{\nu \tau_\beta} u (x', \rho (x'))| \leq C \sqrt{b_\gamma}
\end{equation}
for $1\leq \beta \leq \gamma$.

From \eqref{jw-12}, we see that \eqref{jw-19} holds for $\alpha = n-2$.
In the following, we will prove there exist positive constants $\theta_\alpha$ and $C$ depending only on $\Omega$,
$|\varphi|_{C^{3,1} (\ol \Omega)}$ and $|\tilde{f}|_{C^{1,1} (\ol \Omega)}$ such that if any point $(x', \rho (x')) \in \partial \Omega$
satisfies  $|x_\beta| \leq \theta_\alpha \sqrt{b_\alpha}$ for any $1 \leq \beta \leq \alpha$
and $|x_\beta| \leq \theta_\alpha \frac{b_\alpha}{\sqrt{b_\beta}}$ for any $\alpha + 1 \leq \beta \leq n-1$, we have
\begin{equation}
\label{jw-21}
|\nabla_{\nu \tau_\beta} u (x', \rho (x'))| \leq C \sqrt{b_\alpha}
\end{equation}
for $1\leq \beta \leq \alpha$.

Let
\[
\omega := \{x \in \Omega: |x_\beta| < \delta \frac{b_\alpha}{\sqrt{b_\beta}}, \beta = \alpha + 1, \ldots, n-1,
   x_n < \delta^2 b_\alpha\},
\]
where $\delta$ is a positive  sufficiently small constant independent of $b_\alpha$ to be determined later.
By \eqref{jw-2} and that $u \geq 0$ in $\Omega$, we have
\begin{equation}
\label{jw-15}
|\varphi_{\mu \beta \gamma} (0)| \leq \left\{\begin{aligned}
& C \frac{\sqrt{b_\mu b_\beta b_\gamma}}{b_\alpha}, & \ \mbox{ if } \mu, \beta, \gamma > \alpha  \\
& C \frac{\sqrt{b_\beta b_\gamma}}{\sqrt{b_\alpha}}, & \ \mbox{ if } \mu \leq \alpha; \beta, \gamma > \alpha\\
& C \sqrt{b_\gamma}, & \ \mbox{ if } \mu, \beta \leq \alpha; \gamma > \alpha\\
& C \sqrt{b_\alpha}, & \ \mbox{ if } \mu, \beta, \gamma \leq \alpha
\end{aligned}\right.
\end{equation}
provided $b_\alpha$ is sufficiently small.
By \eqref{taylor} and the definition of $\omega$, we have
\begin{equation}
\label{jw-16}
f (x) \leq C b_{n-1} \cdots b_{\alpha+1} b_\alpha^{k+\alpha-n} \mbox{ in } \omega.
\end{equation}
We then note that
\begin{equation}
\label{jw-47}
u \leq C b_\alpha^2 \mbox{ in } \omega
\end{equation}
by the convexity of $u$.

Let
\[
\ul u (x) = \frac{\sigma}{2} \left(b_\alpha |\hat{x}|^2 + \sum_{\beta = \alpha + 1}^{n-1} b_\beta x_\beta^2\right)
    + \frac{1}{2} K x_n^2 - K^2 b_\alpha x_n,
\]
where $\hat{x} := (x_1, \ldots, x_\alpha)$, $\sigma$ and $K$ are sufficiently small and sufficiently large positive constants
to be chosen later. We then have
\begin{equation}
\label{jw-22}
\sigma_k (\lambda (D^2 \ul u)) \geq \sigma^{k-1} K b_{n-1} \cdots b_{\alpha+1} b_\alpha^{k+\alpha-n}.
\end{equation}

Next, we prove
\begin{equation}
\label{jw-23}
\ul u \leq \frac{1}{2} u
\end{equation}
on $\partial \omega$ by choosing suitable $\delta$, $\sigma$ and $K$. Note that
\[
\partial \omega = \partial_1 \omega \cup \partial_2 \omega \cup \partial_3 \omega,
\]
where $\partial_1 \omega := \partial \omega \cap \partial \Omega$, $\partial_2 \omega := \partial \omega \cap \{x_n = \delta^2 b_\alpha\}$
and $\partial_3 \omega := \partial \omega \cap \{|x_\beta| = \frac{\delta b_\alpha}{\sqrt{b_\beta}}$ for some $\beta = \alpha + 1, \ldots, n-1\}$.

We first consider any point $x \in \partial_1 \omega$. By \eqref{jw-2} and \eqref{jw-15},
\[
u (x) \geq \left(\frac{1}{2} - C \delta\right) \sum_{\beta = \alpha + 1}^{n-1} b_\beta x_\beta^2 - C b_\alpha |\hat{x}|^2
   \geq \frac{1}{4} \sum_{\beta = \alpha + 1}^{n-1} b_\beta x_\beta^2 - C b_\alpha |\hat{x}|^2
\]
provided $\delta$ is sufficiently small. By the uniformly convexity of $\Omega$, we have
\[
\ul u (x) \leq \frac{\sigma}{2} \sum_{\beta = \alpha + 1}^{n-1} b_\beta x_\beta^2 - b_\alpha \left(\frac{1}{2} K^2 x_n -
    \frac{\sigma}{2} |\hat{x}|^2\right) \leq \frac{\sigma}{2} \sum_{\beta = \alpha + 1}^{n-1} b_\beta x_\beta^2 - \frac{K^2}{4} b_\alpha |\hat{x}|^2
\]
if $\sigma$ is sufficiently small. Thus, \eqref{jw-23} holds on $\partial_1 \omega$ provided $\sigma$ is sufficiently small and $K$ is sufficiently large.

Next, for $x \in \partial_2 \omega$
\[
\ul u (x) \leq C \delta^2 b_\alpha^2 \sigma - \frac{K^2 \delta^2 b_\alpha^2}{2} \leq 0 \leq \frac{1}{2} u(x)
\]
provided $\sigma$ is sufficiently small or $K$ is sufficiently large.

For $\partial_3 \omega$, as \cite{GTW99}, we only consider the piece $\partial_3' \omega := \partial \omega \cap \left\{x_{n-1} = \frac{\delta b_\alpha}{\sqrt{b_{n-1}}}\right\}$
and other cases can be handled similarly. We first prove that
\begin{equation}
\label{jw-24}
u (x) \geq \frac{1}{4} \delta^2 b_\alpha^2 \mbox{ on } \partial_3' \omega \cap \{x_n < \epsilon_0 \delta^2 b_\alpha\}
\end{equation}
provided $\epsilon_0$ is sufficiently small. On $\partial_3' \omega \cap \{x_n \geq \epsilon_0 \delta^2 b_\alpha\}$, \eqref{jw-23} holds as on $\partial_2 \omega$. Therefore, we only consider the set $\partial_3' \omega \cap \{x_n < \epsilon_0 \delta^2 b_\alpha\} $.
Fix a point $x = (\hat{x}, \tilde{x}, x_n) \in \partial_3' \omega \cap \{x_n < \epsilon_0 \delta^2 b_\alpha\}$, where $\hat{x} = (x_1, \ldots, x_\alpha)$
and $\tilde{x} = (x_{\alpha+1}, \ldots, x_{n-1})$. We consider two cases.

\textbf{Case 1.} Suppose $\hat{x} = 0$. Let $x_0 = (0, \tilde{x}, \rho (0, \tilde{x})) \in \partial \Omega$. We may assume
$\delta \ll \min_{\alpha+1 \leq \gamma \leq n-1}\theta_{\gamma}$. Let $\tilde{\nu} := \tilde{\nu} (x') =
(- D\rho (x'), 1) \in \mathbb{R}^n$ for $x' \in \mathbb{R}^n$ near the origin.
Note that $|u_\gamma| = |\varphi_\gamma - \rho_\gamma u_n| \leq \sqrt{b_\alpha}$ on $\partial \omega \cap \partial \Omega$
for $1 \leq \gamma \leq n - 1$.
Therefore, $\nu = \frac{\tilde{\nu}}{|\tilde{\nu}|}$ and
for any point $y = (0, \tilde{y}, y_n) \in \partial \omega \cap \partial \Omega$, by our induction hypothesis,
we have
\begin{equation}
\label{jw-25}
\begin{aligned}
|u_{\tilde{\nu} (y')} (y)| \,& \leq |u_{\tilde{\nu} (0)}(0)| \\
    & + \sum_{\beta = \alpha+1}^{n-1} \sup_{\xi \in \partial \omega \cap \partial \Omega} \left(|\nabla_{\tilde{\nu}\tau_\beta} u (\xi)| + \sum_{\gamma=1}^{n-1}|\rho_{\gamma\beta} (\xi) u_\gamma (\xi)|\right) \cdot |y_\beta|
   \leq C b_\alpha
\end{aligned}
\end{equation}
by \eqref{jw-19} and that $D u (0) = 0$. It follows that $|u_\nu (y)| \leq C b_\alpha$. Therefore $|Du(y)|\leq Cb_{\alpha}$ since
$|u_{\tau_\gamma} (y)| = |\varphi_\gamma (y)| \leq C b_\alpha$ for each $1 \leq \gamma \leq n-1$.
Thus, by the convexity of $u$, $x_n<\epsilon_0\delta^2 b_{\alpha}$ and $\hat{x}=0$, we have
\begin{equation}
\label{jw-26}
\begin{aligned}
u (x) \geq u (x_0) - C b_\alpha |x - x_0| = \,& u (x_0) - C b_\alpha (x_n - \rho (0, \tilde{x}))\\
   \geq \,& u (x_0) - C \epsilon_0 \delta^2 b_\alpha^2.
\end{aligned}
\end{equation}
By \eqref{jw-15} and that $x_{n-1} = \frac{\delta b_\alpha}{\sqrt{b_{n-1}}}$, we have
\begin{equation}
\label{jw-29}
\begin{aligned}
u (x_0)& =\varphi(x_0)= \frac{1}{2} \sum_{\beta=\alpha+1}^{n-1} b_\beta x_\beta^2 + \frac{1}{6} \sum_{\xi, \beta, \gamma \geq \alpha+1} \varphi_{\xi \beta \gamma} (0) x_\xi x_\beta x_\gamma
    + O (|\tilde{x}|^4)\\
&\geq \frac{1}{2} \delta^2 b_\alpha^2 - C \delta^3 b_\alpha^2 \geq \frac{7}{16} \delta^2 b_\alpha^2
\end{aligned}
\end{equation}
provided $\delta$ is sufficiently small. Combining with \eqref{jw-26}, we have
\begin{equation}
\label{jw-27}
u (x) \geq \frac{3}{8} \delta^2 b_\alpha^2
\end{equation}
when $\epsilon_0$ is sufficiently small.

\textbf{Case 2.} Suppose $\hat{x} \neq 0$. We may assume $x \notin \partial \Omega$, otherwise \eqref{jw-27} can be derived
similar as \eqref{jw-29}. Let $x_0 = (0, \tilde{x}, x_n^0) = (0, \tilde{x}, \rho (0, \tilde{x})) \in \partial \Omega$
and $P$ be the 2-dimensional plane spanned by $x$ and the straight line through $x_0$ and parallel to the $x_n$-axis.
Note that $P \subset \{x_{n-1} = \frac{\delta b_\alpha}{\sqrt{b_{n-1}}}\}$.
Let $\gamma$ be the intersection of $\partial \Omega$ and $P$. It is clear that $x^* = (0, \tilde{x}, x_n + \varepsilon_1) \in P$,
where $\varepsilon_1 := - \sum_{\beta=1}^\alpha x_\beta \rho_\beta (0, \tilde{x})$. Note that
\[
\begin{aligned}
x_n^* - x_n^0 = \,& x_n - \sum_{\beta=1}^\alpha x_\beta \rho_\beta (0, \tilde{x}) - \rho (0, \tilde{x}) \\
 > \,& \rho (\hat{x}, \tilde{x}) - \sum_{\beta=1}^\alpha x_\beta \rho_\beta (0, \tilde{x}) - \rho (0, \tilde{x}) \geq 0
\end{aligned}
\]
by the convexity of $\rho$.
Suppose $$e_1=\frac{(\hat{x}, 0, - \varepsilon_1)}{\sqrt{|\hat{x}|^2+\varepsilon_1^2}}, \ \ e_2=\frac{(\varepsilon_1 \hat{x}, 0, |\hat{x}|^2)}{|\hat{x}|\sqrt{|\hat{x}|^2+\varepsilon_1^2}}.$$ It is easy to see that $e_1$ and $e_2$ are the unit tangent vector and unit normal vector of the curve $\gamma$
at $x_0$ respectively. By using the coordinate system $\{x_0; e_1, e_2\}$, the curve $\gamma$ can be written by
\begin{equation}
\label{proj}
\xi_2 = \kappa_0 \xi_1^2 + O (|\xi_1|^3) \mbox{ as } \xi_1 \rightarrow 0
\end{equation}
for some positive constant $\kappa_0$ depending only on $\partial \Omega$.
Thus, the straight line from $x$ to $x^*$ meets $\gamma$ at a point $\bar{x}$ satisfying
\begin{equation}
\label{jw-28}
c_0 |x - x^*| \leq |\bar{x} - x^*| \leq |\bar{x} - x|
\end{equation}
for some $0 < c_0 < 1$ depending only on $\partial \Omega$.

Let $\bar{x} = \bar{\xi}_1e_1+\bar{\xi}_2e_2$. Suppose  $\cos\theta = \langle e_2, E_n \rangle$, where
$E_n = (0, \ldots, 0, 1)$ is the direction of the $x_n$-axis.
Thus, we find $$\bar{\xi}_2 = (x_n^* - x_n^0) \cos \theta \text{ and } |\bar{\xi}_1| \leq C \sqrt{\bar{\xi}_2}$$ for some $C$ depending only on $\kappa_0$ by \eqref{proj},
since $|\varepsilon_1| \leq C \epsilon_0 \delta^2 b_\alpha$ and $x_n^* - x_n^0 \leq C \epsilon_0 \delta^2 b_\alpha$ which can be sufficiently  small. Next, we see
\[
E_n = \frac{- \varepsilon_1}{\sqrt{\varepsilon_1^2 + |\hat{x}|^2}} e_1 + \frac{|\hat{x}|}{\sqrt{\varepsilon_1^2 + |\hat{x}|^2}} e_2.
\]
It follows that, using the definition of $\varepsilon_1$,
\[
\bar{x}_n - x_n^0 = \frac{- \varepsilon_1 \bar{\xi}_1 + |\hat{x}| \bar{\xi}_2}{\sqrt{\varepsilon_1^2 + |\hat{x}|^2}}
   \leq |D \rho (0, \tilde{x})| |\bar{\xi}_1| + |\bar{\xi}_2| \leq C \epsilon_0 \delta^2 b_\alpha
\]
and therefore we get
\[
\bar{x}_n \leq C \epsilon_0 \delta^2 b_\alpha + x_n^0 \leq C \epsilon_0 \delta^2 b_\alpha.
\]
Thus, $\bar{x}\in\omega$ if $\epsilon_0$ is sufficiently small.

By \eqref{jw-2} and $\bar{x}\in\omega$, we have
\begin{equation}
\label{jw-30}
\begin{aligned}
u (\bar{x}) &= \frac{1}{2} \sum_{\beta=1}^{n-1} b_\beta \bar{x}_\beta^2 + + \frac{1}{6} \sum_{\xi, \beta, \gamma } \varphi_{\xi \beta \gamma} (0) \bar{x}_\xi \bar{x}_\beta \bar{x}_\gamma
    + O (|\bar{x}'|^4)\\
& \leq u (x_0) + \frac{1}{2} \sum_{\beta=1}^{\alpha} b_\beta \bar{x}_\beta^2
   + C \delta^3 b_\alpha^2.
  \end{aligned}
\end{equation}
By \eqref{BC2-1}, we have
\begin{equation}
\label{jw-31}
\begin{aligned}
\frac{1}{2} \sum_{\beta=1}^{\alpha} b_\beta \bar{x}_\beta^2 \leq C b_\alpha \bar{x}_n + b_\alpha O (|\bar{x}'|^3)
  \leq C b_\alpha \bar{x}_n + C \delta^3 b_\alpha^2.
\end{aligned}
\end{equation}
Since $|\varepsilon_1| \leq C \epsilon_0 \delta^2 b_\alpha$, using the same argument as Case 1, \eqref{jw-26} and \eqref{jw-27}
hold for $x = x^*$ if $\epsilon_0$ is small enough, namely
$$u (x^*) \geq u (x_0) - C \epsilon_0 \delta^2 b_\alpha^2, \ \  \  u (x^*) \geq \frac{3}{8} \delta^2 b_\alpha^2.$$
Combining with \eqref{jw-30}, \eqref{jw-31}, we get
\begin{equation}
\label{jw-32}
\begin{aligned}
u (\bar{x}) \leq \,& u (x^*) +C\epsilon_0\delta^2b_{\alpha}^2+ Cb_{\alpha}\bar{x}_n
   + C \delta^3 b_\alpha^2\\
\leq \,& u (x^*) + C \left(\epsilon_0 + \delta\right) \delta^2 b_\alpha^2 \leq \left(1+\frac{c_0}{3}\right) u (x^*)
\end{aligned}
\end{equation}
provided $\epsilon_0$ and $\delta$ are sufficiently small.
Therefore, by \eqref{jw-28} and the convexity of $u$, we have
\[
u (x) \geq u (x^*) + \frac{1}{c_0} \min\{0, u(x^*) - u (\bar{x})\} \geq \frac{2}{3} u (x^*) \geq \frac{1}{4} \delta^2 b_\alpha^2
\]
and \eqref{jw-24} is proved. Note that
\[
\ul u (x) \leq C \sigma \delta^2 b_\alpha^2 - \frac{K^2}{2} b_\alpha x_n \mbox{ in } \omega.
\]
Thus, \eqref{jw-23} holds when $\sigma$ is sufficiently small and $K$ is sufficiently large. Then
\eqref{jw-23} is valid on $\partial \omega$.

By \eqref{jw-16} and \eqref{jw-22}, we can further fix $K$ sufficiently large such that
\begin{equation}
\label{jw-46}
\sigma_k (\lambda (D^2 \ul u)) \geq \sup_\omega f \mbox{ in } \omega.
\end{equation}
We have constructed a lower barrier $\ul u$ vanishing at the origin with $\ul u_\nu (0) = - K^2 b_\alpha$.

Now we construct a lower barrier for an arbitrary point $x_0 = (x_1^0, \ldots, x_n^0)
\in \partial \Omega \cap \partial \omega_\epsilon$, where
\[
\omega_\epsilon := \{x \in \Omega: |x_\beta| < \epsilon \delta \frac{b_\alpha}{\sqrt{b_\beta}},
    \beta = \alpha + 1, \ldots, n-1, x_n < \epsilon^2 \delta^2 b_\alpha\}
\]
and $\epsilon$ is a positive constant sufficiently small to be determined later. For $x \in \ol \omega$, write $y = x - x_0$ and define
\begin{equation}
\label{add-8}
\ul u (x) = \ul u (y) = \frac{\sigma}{2} \left(b_\alpha |\hat{y}|^2
    + \sum_{\beta = \alpha + 1}^{n-1} b_\beta y_\beta^2\right) + \frac{1}{2} K y_n^2 - K^2 b_\alpha \bar{y}_n
\end{equation}
as above, where
\[
\bar{y}_n := y_n - \sum_{\beta=1}^{n-1} \rho_\beta (x'_0) y_\beta.
\]
Define $l (y')$ to be the linear function
\[
l (y') := u (x_0) + \sum_{\beta = 1}^{n-1} \nabla_\beta u (x_0) y_\beta.
\]
Now we prove $\ul u \leq u - l (y')$ on $\partial \omega$ by choosing suitable positive constants $\sigma$, $K$
and $\epsilon$.

First we consider $x \in \partial_1 \omega$ as before. Let
$\tilde{b}_\beta := \max\{b_\alpha, b_\beta\}$ for $1 \leq \beta \leq n-1$. Thus, by \eqref{jw-2} and \eqref{jw-15},
we have
\[
\frac{\partial^2 \varphi}{\partial x_\beta^2} (x', \rho (x')) \geq b_\beta -  C \delta \tilde{b}_\beta,
  \  \beta = 1, \ldots, n-1
\]
and
\[
\left|\frac{\partial^2 \varphi}{\partial x_\beta \partial x_\mu} (x', \rho (x'))\right| \leq C \delta
  \sqrt{\tilde{b}_\beta \tilde{b}_\mu},\  \beta, \mu = 1, \ldots, n-1; \beta \neq \mu
\]
for any $(x', \rho (x')) \in \partial_1 \omega$. Hence by Taylor's expansion of $u (x', \rho(x'))
= \varphi (x', \rho (x'))$ at $x_0$, we obtain
\begin{equation}
\label{jw-40}
u (x) \geq l (y') + \frac{1}{4} \sum_{\beta=1}^{n-1} b_\beta y_\beta^2 - C b_\alpha |\hat{y}|^2
\end{equation}
provided $\delta$ is sufficiently small.
By Taylor's expansion of $\rho$ at $x'_0$ and the uniform convexity of $\Omega$, we get
\[
y_n = \rho (x) - \rho (x_0) \geq \sum_{\beta=1}^{n-1} \rho_\beta (x'_0) y_\beta + \kappa |y'|^2,
\]
where $\kappa$ is a positive constant depending only on the principal curvatures of $\partial \Omega$.
It follows that
\begin{equation}
\label{jw-41}
\bar{y}_n \geq \kappa |y'|^2 \mbox{ on } \partial_1 \omega.
\end{equation}
Next, since $|\rho_\beta (x'_0)| \leq C |x'_0| \leq C \epsilon \delta \sqrt{b_\alpha}$
for any $\beta = 1, \ldots, n-1$, we have
\begin{equation}
\label{jw-45}
y_n^2 = \left(\bar{y}_n + \sum_{\beta = 1}^{n-1} \rho_\beta (x'_0) y_\beta\right)^2
  \leq 2 \bar{y}_n^2 + C \epsilon^2 \delta^2 b_\alpha |y'|^2.
\end{equation}
Therefore, by \eqref{jw-41}, \eqref{jw-45} and $|\bar{y}_n| \leq C b_\alpha$, we have
\begin{equation}
\label{jw-42}
\begin{aligned}
\ul u (x) \leq \,& \frac{\sigma}{2} \sum_{\beta = \alpha + 1}^{n-1} b_\beta y_\beta^2
   - b_\alpha \left(\frac{1}{2} K^2 \bar{y}_n - \frac{\sigma}{2} |\hat{y}|^2 - C \epsilon^2 \delta^2 |y'|^2\right)\\
     \leq \,& \frac{\sigma}{2} \sum_{\beta = \alpha + 1}^{n-1} b_\beta y_\beta^2 - \frac{K^2}{4} b_\alpha |y'|^2
\end{aligned}
\end{equation}
provided $\sigma$ is sufficiently small and $K$ is sufficiently large.
Combing \eqref{jw-40} and \eqref{jw-42} we obtain $u (x) - l (y') \geq \ul u (x)$ if $K$ is large enough.

For $x \in \partial_2 \omega$, we note that $\bar{y}_n \geq (1 - \epsilon^2 - C \epsilon) \delta^2 b_\alpha$ since
$x_0 \in \partial \Omega \cap \partial \omega_\epsilon$.
By \eqref{jw-2} and \eqref{jw-15} again, we find
\begin{equation}
\label{jw-44}
|l (y')| \leq C (\epsilon^2 + \epsilon) \delta^2 b_\alpha^2.
\end{equation}
It follows that
\[
\begin{aligned}
\ul u (x) + l (y') \leq \,& C \sigma \delta^2 b_\alpha^2 + CK \delta^2 b_\alpha^2
   - K^2 (1 - \epsilon^2 - C \epsilon) \delta^2 b_\alpha^2
     + C (\epsilon^2 + \epsilon) \delta^2 b_\alpha^2\\
     < \,& 0 \leq u (x)
\end{aligned}
\]
if $K$ is sufficiently large and $\epsilon$ is sufficiently small.

Now we consider $x \in \partial_3 \omega$. We only need to consider the case $x \in \partial'_3 \omega$
as before. First we note that if $x_n \geq \epsilon_0 \delta^2 b_\alpha$, where $\epsilon_0$ is the constant defined by \eqref{jw-24},
\[
\bar{y}_n \geq (\epsilon_0 - \epsilon^2 - C \epsilon) \delta^2 b_\alpha.
\]
Thus, we have $\ul u (x) + l (y') < 0 < u (x)$ if $\epsilon \ll \epsilon_0$ and $K$ is large enough as in the case $x \in \partial_2 \omega$.

Now we consider $x \in \partial_3' \omega \cap \{x_n < \epsilon_0 \delta^2 b_\alpha\}$. By \eqref{jw-24} and \eqref{jw-44}, we have
\[
u (x) - l(y') \geq \left(\frac{1}{4} - C (\epsilon + \epsilon^2)\right) \delta^2 b_\alpha^2
    \geq \frac{1}{8} \delta^2 b_\alpha^2
\]
provided $\epsilon$ is sufficiently small. By \eqref{jw-45}, we have
\[
\ul u (x) \leq C (\sigma + \epsilon) \delta^2 b_\alpha^2 - \frac{K^2}{2} b_\alpha \bar{y}_n.
\]
It follows that $\ul u (x) \leq u (x) - l(y')$ if $\sigma$ and $\epsilon$ are small enough and $K$ is large enough.
We can further fix $K$ sufficiently large such that \eqref{jw-46} holds.

To construct an upper barrier for $x_0 \in \partial \Omega \cap \partial \omega_\epsilon$, we define
\[
\ol u (x) = \ol u (y) := M \left(\frac{1}{2} b_\alpha |\hat{y}|^2
   + \frac{1}{2} \sum_{\beta = \alpha+1}^{n-1} b_\beta y_\beta^2 + b_\alpha \bar{y}_n\right),
\]
where $M$ is a sufficiently large constant to be determined later.
By \eqref{jw-47}, we can fix $\epsilon$ sufficiently small and $M$ sufficiently large such that
\[
u - l (y') \leq \ol u \mbox{ on } \partial \omega.
\]
Furthermore, we find
\[
\det (D^2 \ol u) = 0 \mbox{ in } \omega.
\]
By the convexity of $u$ and the maximal principle, we have
\[
u - l (y') \leq \ol u \mbox{ in } \omega.
\]
Thus, we have
\begin{equation}
\label{add-4}
|u_\nu| \leq C b_\alpha \mbox{ on } \partial_1 \omega_\epsilon.
\end{equation}
Similar to \eqref{plus1}, we find, by \eqref{BC2-1}, \eqref{jw-2}, \eqref{add-4}
and the convexity of $u$,
\begin{equation}
\label{add-5}
\begin{aligned}
\sup_{\ol \omega} |u_\beta| \leq \,& \sup_{\partial \omega \cap \partial \Omega} |\varphi_\beta - u_n \rho_\beta| \leq C b_{\alpha}^{3/2}, \mbox{ for } 1 \leq \beta \leq \alpha,\\
\sup_{\ol \omega} |u_\beta| \leq \,& \sup_{\partial \omega \cap \partial \Omega} |\varphi_\beta - u_n \rho_\beta| \leq C \sqrt{b_\beta} b_{\alpha}, \mbox{ for } \alpha+ 1 \leq \beta \leq n-1.
\end{aligned}
\end{equation}
As \eqref{add-6} and \eqref{add-7}, for any $x = (x', \epsilon^2 \delta^2 b_\alpha) \in \partial_2 \omega_{\epsilon}$,
we have
\[
- \epsilon^2 \delta^2 b_{\alpha} u_n (x) \leq u(0) - u(x) + \sum_{\beta = 1}^{n-1} x_\beta u_\beta (x) \leq C b_\alpha^2
\]
and
\[
\delta^2 (1 - \epsilon^2) b_{\alpha} u_n (x) \leq u(y) - u(x) \leq C b_{\alpha}^2, \mbox{ with } y = (x', \delta^2 b_\alpha) \in \partial_2 \omega
\]
by \eqref{jw-47} and \eqref{add-5}. It follows that $|u_n| \leq C b_\alpha$ on $\partial_2 \omega_{\epsilon}$. Therefore,
by the convexity of $u$ again, we can also get a bound $|u_n| \leq C b_\alpha$ on $\partial_3 \omega_{\epsilon}$.
We then obtain
\begin{equation}
\label{jw-33}
|D u| \leq C b_\alpha \mbox{ on } \overline{\omega}_\epsilon.
\end{equation}
To proceed we consider $\nabla_\beta (u - \varphi)$ for $1 \leq \beta \leq \alpha$.
By \eqref{jw-15}, \eqref{BC2-1}, \eqref{add-5} and \eqref{jw-33}, we have
\begin{equation}
\label{jw-34}
|\nabla_\beta (u-\varphi)| \leq C b_\alpha^{3/2} \mbox{ on } \ol \omega_\epsilon.
\end{equation}
By \eqref{jw-15} again, we find
\begin{equation}
\label{jw-35}
|F^{ij} (\nabla_\beta \varphi)_{ij}| \leq C \sqrt{b_\alpha} \left(\sum_{i=1}^{\alpha} F^{ii} + \sum_{i=\alpha+1}^{n-1} \frac{b_i}{b_\alpha} F^{ii}\right)
   \mbox{ in } \omega_\epsilon.
\end{equation}
Let
\[
m := b_{n-1} \cdots b_{\alpha+1} b_\alpha^{k+\alpha-n}.
\]
We find
\[
m^{-1/2(k-1)} \leq (b_\alpha^{k-1})^{-1/2(k-1)} = b_\alpha^{-1/2}.
\]
By \eqref{smooth}, \eqref{jw-16} and \eqref{jw-33}, we have
\begin{equation}
\label{jw-36}
\begin{aligned}
|F^{ij} (\nabla_\beta u)_{ij}| \leq \,& |\nabla_\beta f^{1/k}| + C \left(f^{1/k} + b_\alpha \sum_{i=1}^n F^{ii}\right)\\
  \leq \,& C m^{1/k} m^{-1/2(k-1)} + C b_\alpha \sum_{i=1}^n F^{ii}\\
   \leq \,& C b_\alpha^{-1/2} m^{1/k} + C b_\alpha \sum_{i=1}^n F^{ii}
\end{aligned}
\end{equation}
in $\omega_\epsilon$. Now we consider an arbitrary point $x_0 = (x_1^0, \ldots, x_n^0)
\in \partial \Omega \cap \partial \omega_{\epsilon_1}$ with $\epsilon_1 \ll \epsilon$ to be
determined later. Let $\ul u$ be the function defined in \eqref{add-8}.
By the concavity of $F$, we have
\begin{equation}
\label{jw-37}
\frac{1}{\sqrt{b_\alpha}} F^{ij} (\ul u - u)_{ij} \geq c_0 \left(\sqrt{b_\alpha}\sum_{i=1}^{\alpha} F^{ii} +
  \sqrt{b_\alpha} \sum_{i=\alpha+1}^{n-1} \frac{b_i}{b_\alpha} F^{ii} + F^{nn}
   + \frac{m^{1/k}}{\sqrt{b_\alpha}}\right)
\end{equation}
in $\omega_\epsilon$ for some positive constant $c_0$ which may depend on $\sigma$ and $K$. Let
\[
v = \frac{1}{\sqrt{b_\alpha}} \sum_{\beta=\alpha+1}^{n-1} b_\beta y_\beta^2 + \sqrt{b_\alpha} \bar{y}_n.
\]
Thus, for any $x \in \partial_2 \omega_\epsilon$, we have
\[
v (x) \geq \sqrt{b_\alpha} \bar{y}_n \geq \sqrt{b_\alpha} (\epsilon^2 - \epsilon_1^2 - C \epsilon_1) \delta^2 b_\alpha
  \geq \frac{\epsilon^2}{2} \delta^2 b_\alpha^{3/2}
\]
if we let $\epsilon_1$ be small enough such that $\epsilon_1^2 + C \epsilon_1 \leq \epsilon/2$.
Next, for any $x \in \partial_3 \omega_\epsilon$, which means $|x_{\beta_0}| = \epsilon \delta b_\alpha/\sqrt{b_{\beta_0}}$
for some $\alpha + 1 \leq \beta_0 \leq n-1$, we find
\[
\begin{aligned}
v (x) \geq \,& \frac{1}{\sqrt{b_\alpha}} \sum_{\beta=\alpha+1}^{n-1} b_\beta y_\beta^2
   \geq \frac{1}{\sqrt{b_\alpha}} b_{\beta_0} y_{\beta_0}^2
      \geq \frac{b_{\beta_0}}{\sqrt{b_\alpha}} \left(\frac{1}{2} x_{\beta_0}^2 - (x_{\beta_0}^0)^2\right)\\
      \geq \,& \left(\frac{\epsilon}{2} - \epsilon_1^2\right) \delta^2 b_\alpha^{3/2}
      \geq \frac{\epsilon}{2} \delta^2 b_\alpha^{3/2}
\end{aligned}
\]
provided $\epsilon_1^2 \leq \epsilon/2$. Combining the above two inequalities and \eqref{jw-34}, we have
\[
|\nabla_\beta (u-\varphi)| \leq C v \mbox{ on } \partial \omega_\epsilon.
\]
Thus, by \eqref{jw-35}, \eqref{jw-36} and \eqref{jw-37}, there exist positive constants $A \gg B \gg 1$ such that
\[
\begin{aligned}
F^{ij} \left(\frac{A}{\sqrt{b_{\alpha}}} (\ul u - u) - B v \pm \nabla_\beta (u - \varphi)\right)_{ij} \geq \,& 0 \mbox{ in } \omega_\epsilon\\
  \frac{A}{\sqrt{b_{\alpha}}} (\ul u - u) - B v \pm \nabla_\beta (u - \varphi) \leq \,& 0 \mbox{ on } \partial \omega_\epsilon
\end{aligned}
\]
By the maximal principle we have
\[
\frac{A}{\sqrt{b_{\alpha}}} (\ul u - u) - B v \pm \nabla_\beta (u - \varphi) \leq 0 \mbox{ in } \omega_\epsilon
\]
and
\[
|\nu (\nabla_\beta u) (x_0)| \leq C \sqrt{b_{\alpha}}, \mbox{ for } 1\leq \beta \leq \alpha.
\]
We thus have proved
\[
|\nabla_{\nu \tau_\beta} u| \leq C \sqrt{b_{\alpha}} \mbox{ on } \partial \omega_{\epsilon_1} \cap \partial \Omega,
  \mbox{ for } 1\leq \beta \leq \alpha.
\]
By induction, \eqref{jw-18} is proved.

At the origin, we have
\begin{equation}
\label{jw-39}
\sigma_{k-1} (b) u_{nn} (0) - \sum_{\alpha \leq n-1} u_{n\alpha}^2 (0) \sigma_{k-2;\alpha} (b) + \sigma_k (b) = f(0).
\end{equation}
Note that, by \eqref{jw-18},
\begin{equation}
\label{jw-38}
\begin{aligned}
& \frac{\sum_{\alpha \leq n-1} u_{n\alpha}^2 \sigma_{k-2;\alpha} (b)}{\sigma_{k-1} (b)}\\
  \leq \,& C\frac{\sum_{\alpha =1}^{n-k} u_{n\alpha}^2 b_{n-k+2}\cdots b_{n-1}}{b_{n-k+1}\cdots b_{n-1}}
    + C\frac{\sum_{\alpha =n-k+1}^{n-1} u_{n\alpha}^2 b_{n-k+1}\cdots \hat{b}_\alpha \cdots b_{n-1}}{b_{n-k+1}\cdots b_{n-1}}\\
  \leq \,& C \frac{b_{n-k+1}\cdots b_{n-1}}{b_{n-k+1}\cdots b_{n-1}} = C.
\end{aligned}
\end{equation}
Combining \eqref{cor-1}, \eqref{jw-39} and \eqref{jw-38}, we obtain
\[
u_{nn} (0) \leq C.
\]

\bigskip

{\bf Acknowledgement.} The authors wish to thank Professor Bo Guan for many helpful discussions.

\end{document}